\documentclass[reqno,12pt]{amsart}

\usepackage[utf8]{inputenc}

\usepackage{enumerate}
\usepackage[margin=1in,marginparwidth=0.7in]{geometry}

\usepackage{tikz-cd}
\usepackage{ifpdf}
\usepackage{amsmath}
\usepackage{amsfonts}
\usepackage{amssymb}
\usepackage{amsthm}
\usepackage{comment}
\usepackage[ocgcolorlinks,hyperfootnotes=false,colorlinks=true,citecolor=blue,linkcolor=blue,urlcolor=blue]{hyperref}
\usepackage{setspace}
\usepackage{amsrefs}
\usepackage{nicefrac}
\usepackage{graphicx}
\usepackage{color}
\usepackage{mathtools}

\newcommand{\ignore}[1]{}

\newcommand{\Aut}{\operatorname{Aut}}

\newcommand{\sabs}[1]{\lvert {#1} \rvert}

\newcommand{\snorm}[1]{\lVert {#1} \rVert}

\newcommand{\C}{{\mathbb{C}}}

\newcommand{\Z}{{\mathbb{Z}}}
\newcommand{\N}{{\mathbb{N}}}

\newcommand{\bB}{{\mathbb{B}}}

\newcommand{\bR}{{\mathbb{R}}}

\newtheorem{thm}{Theorem}[section]

\newtheorem{prop}[thm]{Proposition}

\newtheorem{cor}[thm]{Corollary}    

\newtheorem{lemma}[thm]{Lemma}

\theoremstyle{definition}
\newtheorem{defn}[thm]{Definition}
\newtheorem{example}[thm]{Example}

\theoremstyle{remark}
\newtheorem{remark}[thm]{Remark}

\newcommand{\avoidbreak}{\postdisplaypenalty=100}

\author{Dusty Grundmeier}
\address{Department of Mathematics, The Ohio State University,
Columbus, OH 43210, USA}
\email{grundmeier.1@osu.edu}

\author{Ji\v{r}\'{\i} Lebl}
\address{Department of Mathematics, Oklahoma State University,
Stillwater, OK 74078, USA}
\email{lebl@okstate.edu}

\date{\today}
    
\title{Rational Maps of Balls and their Associated Groups}


\begin{document}

\begin{abstract}
Given a proper, rational map of balls, D'Angelo and Xiao introduced five
natural groups encoding properties of the map. We study these groups using a
recently discovered normal form for rational maps of balls. Using this
normal form, we also provide several new groups associated to the map. 
\end{abstract}

\maketitle



\section{Introduction} \label{section:intro}

Suppose $f \colon \bB_n \to \bB_N$ is a rational proper map,
$f = \frac{p}{g}$ where $p \colon \C^n \to \C^N$ and
$g \colon \C^n \to \C$ are polynomials, $\frac{p}{g}$
is in lowest terms,
and $g(0) = 1$.  Further, suppose $f$ is of degree $d$.
We say that $f$ is in \emph{normal form} if $f(0) = p(0) = 0$,
\begin{equation*}
g = 1 + g_2 + \cdots + g_{d-1} 
\quad
\text{and}
\quad
g_2(z) = \sum_{\ell=1}^n \sigma_\ell z_\ell^2
\quad
0 \leq \sigma_1 \leq \cdots \leq \sigma_n ,
\end{equation*}
where $g_k$ are homogeneous polynomials of degree $k$.  That is, 
$f$ is in normal form if it fixes the origin, the denominator has no linear
terms, and the quadratic terms are diagonalized.
In \cite{Lnormal} the second author proved that every rational proper map of balls 
can be put into the normal form above via spherical equivalence and that this
is a normal form up to unitary automorphisms where the source
automorphism fixes $g_2$.

A natural problem is to study holomorphic maps invariant under subgroups
of the automorphism group. This problem has a long history, going back
to at least Cartan in \cite{C1}. We refer the reader to
\cites{DX1, DX2, CatlinDAngelo, MRC,DL, F1,F2, G1, Rudin} and especially
D'Angelo's books \cites{JPDrationalbook,JPDhypersurfaces,JPDhermitian}.
Rudin \cite{Rudin} and Forstneri{\v c} \cite{F1} studied invariant proper
maps to general domains in $\mathbb{C}^n$. In this context, one can
construct proper maps invariant under any finite subgroup of the unitary
group $U(n)$. On the other hand, if the target domain is required to be
a unit ball, then the problem has more structure. Forstneri{\v c}~\cite{F3}
showed that if the map is sufficiently smooth up to the boundary, then
the map must be rational, and furthermore, in \cite{F2}, he showed that
for most groups, it is not possible to find a proper, rational map of
balls invariant under that group. In \cite{DL}, D'Angelo and Lichtblau
gave the decisive result, completely characterizing which groups admit invariant, 
proper, rational maps between balls.
 
Moreover, given a proper map of balls, one can use various groups to
detect various properties of the map.
In particular, D'Angelo and Xiao~\cites{DX1,DX2} introduced the
groups $A_f$, $G_f$, $\Gamma_f$,
$T_f$, and $H_f$, defined below.  Using the normal form, 
we add several new groups, $D_f$, $\Sigma_f$, $\Delta^{(a,b)}_f$.

\begin{defn}
\pagebreak[2]
Suppose $f \colon \bB_n \to \bB_N$ is a rational proper map.
\begin{enumerate}[(i)]
\item
$A_f$ is the subgroup of $\Aut(\bB_n) \oplus \Aut(\bB_N)$ such that
$(\varphi,\tau) \in A_f$ if $\tau \circ f = f \circ \varphi$.
\item
$\Gamma_f$ is the subgroup of $\Aut(\bB_n)$ such that
$\varphi \in \Gamma_f$ if there is a $\tau \in \Aut(\bB_N)$
such that $\tau \circ f = f \circ \varphi$.
\item
$G_f$ is the subgroup of $\Aut(\bB_n)$ such that
$\varphi \in G_f$ if $f = f \circ \varphi$.
\item
$T_f$ is the subgroup of $\Aut(\bB_N)$ such that
$\varphi \in T_f$ if there is a $\varphi \in \Aut(\bB_n)$
such that $\tau \circ f = f \circ \varphi$.
\item
$H_f$ is the subgroup of $\Aut(\bB_N)$ such that
$\tau \in H_f$ if $\tau \circ f = f$.
\end{enumerate}
\end{defn}

Given a polynomial $\rho(z, \bar{z})$, let $\rho_{(a,b)}$ denote the monomials of $\rho$ that are of degree $a$ in the holomorphic $z$ variables, and degree $b$ in the antiholomorphic $\bar{z}$
variables. We refer $\rho_{(a,b)}$ as the bidegree-$(a,b)$ part of $\rho$, and we write $*$ instead of $a$ or $b$ to denote all degrees together.

\begin{defn} Let $f=\frac{p}{g}\colon \bB_n \to \bB_N$ be a proper, rational map in normal form.
\begin{enumerate}[(i)]
\item
$D_f$ is the subgroup of $U(n)$ such that $U \in D_f$ when $g \circ U = g$.
\item
$\Sigma_f$ is the subgroup of $U(n)$ such that $U \in \Sigma_f$ when $g_2 \circ U = g_2$.
\item
$\Delta^{(a,b)}_f$ is the subgroup of $U(n)$ such that $U \in \Delta^{(a,b)}_f$ when
$$\bigl(\sabs{g(z)}^2-\snorm{p(z)}^2\bigr)_{(a,b)} =
\bigl(\sabs{g(Uz)}^2-\snorm{p(Uz)}^2\bigr)_{(a,b)}.$$

\end{enumerate}
\end{defn}

All the groups are closed, and therefore Lie subgroups. D'Angelo--Xiao~\cite{DX2} showed this for the groups $A_f,\Gamma_f,G_f,T_f,H_f$, and it is immediate for the new groups we defined.
As long as they are compact, it follows from standard theory that they can be conjugated to
a subgroup of the unitary group.  In fact, $\Gamma_f$ is noncompact
if and only if $f$ is linear fractional (an automorphism if $n=N$), see~\cite{DX2}.
Moreover, Lichtblau~\cite{Li} (see also D'Angelo--Lichtblau~\cite{DL}) has shown
that $G_f$ must be finite, fixed-point-free, and cyclic.
In the present paper, we prove is that once the map is in normal form, the relevant groups are all subgroups of the unitary group.
Note that if $f$ is in normal form and it is linear fractional, then it is in fact linear.

\begin{thm} \label{thm1}
Suppose $f \colon \bB_n \to \bB_N$ is a rational proper map in normal form
and $f$ is not linear.  Then
\begin{enumerate}[(i)]
\item
$A_f \leq U(n) \oplus U(N)$ is a closed subgroup.
\item
$G_f \leq \Gamma_f \leq D_f \leq \Sigma_f \leq U(n)$
and
$\Gamma_f \leq \Delta^{(a,b)}_f$ 
are all closed subgroups.
\item
$H_f \leq U(N)$ and $T_f \leq U(N)$ are closed subgroups.
\end{enumerate}
\end{thm}

In particular, for a mapping that is not an automorphism, 
once the mapping is in normal form, all the groups are subgroups of the
unitary group.  One motivation for introducing the groups $\Sigma_f$, $D_f$,
and $\Delta^{(a,b)}_f$ is that
they are often easier to compute.
If the $\sigma$ invariants are nonzero and distinct, namely,
$0 < \sigma_1 < \ldots < \sigma_n$, then $\Sigma_f \cong (\Z_2)^n$.
Therefore, a corollary is that $G_f$ must be cyclic and
fixed-point-free for such $f$.  In general, standard linear algebra says that
$\Sigma_f$ is a direct sum of groups where if we have $k$ zero $\sigma$'s,
the first factor is $U(k)$, and for each set of $k$ nonzero equal $\sigma$s
we get a factor of $O(k)$ (real orthogonal group).  In particular, when all $\sigma$s are
nonzero and distinct we get a direct sum of $O(1)=\{1,-1\}$.

\begin{cor} \label{cor1}
Suppose $f = \frac{p}{g} \colon \bB_n \to \bB_N$
is a rational proper map in normal form
such that $0 < \sigma_1 < \ldots < \sigma_n$.
Then $G_f$ is either trivial, or $G_f = \{ I, -I \}$.
In particular,
if $G_f = \{ I, -I \}$, then $p(z) = p(-z)$ and $g(z)=g(-z)$, and the degree of $f$ is at least 4.

Furthermore, for any sufficiently small
$0 < \sigma_1 < \ldots < \sigma_n$, there exists a degree 4 map
$f = \frac{p}{g} \colon \bB_n \to \bB_N$ 
where $g(z) =  1+\sigma_1 z_1^2 + \cdots + \sigma_n z_n^2$ and
such that $G_f = \{ I, -I \}$.
\end{cor}

In particular, the corollary says that when $\sigma_1, \ldots, \sigma_n$ are
nonzero and distinct, then for $G_f$ to be nontrivial, the map $f$ must have degree at least 4.  Hence, when the degree of $f$ is 3, we have that $G_f$
is trivial, but we then note that $D_f = \Sigma_f \cong (\Z_2)^n$.
It is not difficult to find a map of degree 4 or higher
with a denominator of the form
\begin{equation}
1+\sigma_1 z_1^2 + \cdots + \sigma_n z_n^2
+ \epsilon_1 z_1^3+\cdots+\epsilon_n z_n^3
\avoidbreak
\end{equation}
for small enough $\sigma$'s and $\epsilon$'s
(see Proposition~\ref{prop:constructwithdenom} for example),
in which case $D_f = \{ I \}$.

A natural problem is to determine what properties of the map $f$ can be detected using the above groups. In particular, we focus on $\Gamma_f$. 
 D'Angelo and Xiao~\cite{DX1} proved that a rational
proper map of balls whose $\Gamma_f$ group is not compact
is spherically equivalent to the linear embedding and therefore is linear
fractional, in which case $\Gamma_f = \Aut(\bB_n)$.
Gevorgyan, Wang, and Zimmer~\cite{Z}
extended this result to maps that are only $C^2$
up to the boundary.
D'Angelo and Xiao~\cite{DX1} further proved that $\Gamma_f$ contains the torus
if and only if $f$ is spherically equivalent
to a monomial map (a map where each component is a single monomial).
They also prove that if the group $\Gamma_f$ contains the center of $U(n)$
(the circle group $\{ e^{i\theta} I \}$), then the map $f$ is spherically equivalent to a polynomial map.
They also show that for any finite subgroup $\Gamma$, there exists a rational $f$
such that $\Gamma_f = \Gamma$.
We give a slightly stronger version of this result in the next theorem. We say a subgroup $\Gamma \leq U(n)$ is defined by finitely many invariant polynomials if there exists polynomials $\rho_1,\ldots, \rho_\ell$ such that $$\Gamma = \bigl\{ U \in U(n) : \rho_j(Uz,\overline{Uz}) = \rho_j(z,\bar{z}), j=1,\ldots, \ell
\bigr\}.$$


\begin{thm} \label{thm:char}
Suppose that $f \colon \bB_n \to \bB_N$ is a rational proper map in normal
form and $f$ is not linear.  Then $\Gamma_f \leq U(n)$ is a subgroup that is
defined by a single invariant polynomial.

Conversely, given any group $\Gamma \leq U(n)$ that is defined by finitely many
invariant polynomials, there exists a rational, proper map
$f \colon \bB_n \to \bB_N$ such that $\Gamma_f = \Gamma$.  Moreover, this
map $f$ can be chosen to be a polynomial that takes the origin to origin,
and hence in normal form.
\end{thm}


We conclude the introduction by outlining the results of the paper. In Section \ref{section:hermitian}, we briefly recall the usual background on Hermitian forms. In Section \ref{section:subgroupsofunit}, we give a sequence of lemmas to prove Theorem \ref{thm1}. In Section \ref{section:charGammaf}, we prove Theorem \ref{thm:char} and characterize the possible $\Gamma_f$ groups. In Section \ref{section:constructions}, we consider the problem of constructing maps with a given denominator. Finally, in Section \ref{section:gensigmas}, we prove Corollary, \ref{cor1} and give an example.


\section{Hermitian forms}
\label{section:hermitian}

In this section, we briefly recall the standard setup to treat real-valued polynomials as Hermitian forms (see Chapter 1 of \cite{JPDrationalbook} for more details on this approach).
A real-valued polynomial in $\C^n$ can be written as a polynomial
\begin{equation}
r(z,\bar{z}) = \sum_{\alpha\beta} c_{\alpha\beta} z^\alpha \bar{z}^\beta ,
\end{equation}
where $\alpha$ and $\beta$ are multiindices. The coefficients
$c_{\alpha\beta}$ can be put into a matrix $[c_{\alpha\beta}]_{\alpha,\beta}$
by putting an order on the monomials (and hence the multiindices) and
having $\alpha$ refer to rows and $\beta$ to columns.  The matrix $[c_{\alpha\beta}]_{\alpha,\beta}$
is called the \emph{matrix of coefficients} of $r$. The polynomial $r$ is real-valued if and only if the matrix $[c_{\alpha\beta}]_{\alpha,\beta}$ is Hermitan. 
Diagonalizing the matrix of coefficients then yields
\begin{equation}
r(z,\bar{z}) = \snorm{P(z)}^2-\snorm{G(z)}^2
\end{equation}
where $\snorm{\cdot}$ denotes the standard Hermitian norm, and $P \colon
\C^n \to \C^a$ and $G \colon \C^n \to \C^b$ are polynomials whose components
are linearly independent.  As the rank is $a+b$, this expansion cannot be
done with fewer than $a+b$ polynomials.  In particular, if the matrix of coefficients of $r$
is positive semidefinite, then $r(z,\bar{z}) = \snorm{P(z)}^2$, and
conversely, if $r$ is a Hermitian sum of squares, its matrix is positive
semidefinite.
The polynomial $P$ in $r(z,\bar{z})=\snorm{P(z)}^2$
only needs to use
those monomials that correspond to nonzero rows
or columns of the matrix.  Hence, if the entries in the matrix corresponding
to purely holomorphic or purely antiholomorphic terms are all zero, then
this corresponds to the row and column corresponding to the monomial $1$.
In other words, in this case, $P$ can be chosen to have no constant
term, that is, $P(0)=0$.


\section{Groups are subgroups of the unitaries}
\label{section:subgroupsofunit}

Given a proper rational map
$\frac{p}{g} \colon \bB_n \to \bB_N$, consider
its \emph{underlying form}
\begin{equation}
r(z,\bar{z}) = \sabs{g(z)}^2-\snorm{p(z)}^2 .
\end{equation}
Since rescaling $r$ by a positive constant does not change the corresponding map,
we will generally normalize $r$ so that $r(0,0)=1$.
The following observation was made in \cite{L2011},
and also used later by D'Angelo--Xiao~\cite{DX1}.
That is, two maps differ by a target automorphism if and only if
the underlying forms are the same (up to rescaling).

\begin{lemma}[Lemma 2.1 from \cite{Lnormal}] \label{lemma:targetautform}
Suppose
$\frac{p}{g} \colon \bB_n \to \bB_N$ 
and $\frac{P}{G} \colon \bB_n \to \bB_N$ are proper rational maps
written in lowest terms such that
$\sabs{g(0)}^2-\snorm{p(0)}^2 = 1$ and
$\sabs{G(0)}^2-\snorm{P(0)}^2 = 1$.
Then there exists a $\tau \in \Aut(\bB_N)$ such that
\begin{equation}
\tau \circ \frac{p}{g} = \frac{P}{G}
\qquad
\text{if and only if}
\qquad
\sabs{g(z)}^2-\snorm{p(z)}^2 = 
\sabs{G(z)}^2-\snorm{P(z)}^2 .
\end{equation}
\end{lemma}

The lemma says that the group $\Gamma_f$
is precisely the group that leaves the underlying form 
unchanged up to rescaling.  We prove Theorem~\ref{thm1} in stages using the following lemmas.

We briefly recall some properties of automorphisms and the normal form from \cite{Lnormal}. Let $\alpha \in \bB^n$ and $t=\sqrt{1-\snorm{\alpha}^2}$. Then every automorphism of $\bB_n$ is of the form $U\varphi_{\alpha}$ where $U \in U(n)$ and 
\begin{equation} \label{e:auto}
    \varphi_{\alpha}(z)=\frac{\alpha-L_{\alpha}z}{1-\langle z, \alpha \rangle},
\quad \text{ and }\quad
    L_\alpha z= \frac{\langle z, \alpha \rangle}{t+1}\alpha + t z.
\end{equation} See Chapter 1 of \cite{JPDrationalbook} for background on automorphisms of $\bB_n$.
Now, we define the $\Lambda_f: \bB_n \to \bR$ function by
\begin{equation}
\Lambda_f(z,\bar{z}) =
\frac{\sabs{g(z)}^2-\snorm{p(z)}^2}{(1-\snorm{z}^2)^d}
=
\frac{r(z,\bar{z})}{(1-\snorm{z}^2)^d}.
\end{equation} In \cite{Lnormal}, the second author proves the following properties of the $\Lambda_f$ function.

\begin{lemma}[From \cite{Lnormal}]
    \label{lemma:Lnormal} If $f=\frac{p}{g}: \bB_n \to \bB_N$ is a rational proper map of degree $d>1$, in lowest terms, then the following hold.
    \begin{enumerate}
        \item If $\tau \in \Aut(\bB_N)$, then $\Lambda_f=\Lambda_{\tau \circ f}$.
        \item If $\psi \in \Aut(\bB_n)$, then $\Lambda_{f} \circ \psi= C \Lambda_{f \circ \psi}$.
        \item The $\Lambda$-function has a unique critical point in $\bB_n$.
        \item There exists $\tau \in \Aut(\bB_N)$ such that $\tau \circ f \circ \varphi_\alpha$ takes the origin to the origin and has no linear terms in its denominator when written in lowest terms if and only if $\alpha \in \bB_n$ is the critical point of the corresponding $\Lambda$-function.
    \end{enumerate}
\end{lemma}

\begin{lemma} \label{lemma:l1thm1}
Suppose $f \colon \bB_n \to \bB_N$ is a rational proper map in normal form
and $f$ is not linear.  Then
$A_f \leq U(n) \oplus U(N)$ is a closed subgroup.
\end{lemma}

\begin{proof}
Suppose $f = \frac{p}{g}$ is a rational proper map in normal form that is
not linear and let
$r(z,\bar{z}) = \sabs{g(z)}^2-\snorm{f(z)}^2$.
Suppose that $(\varphi,\tau) \in A_f$, thus
$\tau \circ f \circ \varphi^{-1} = f$.
Write $f \circ \varphi^{-1}
=
F(z)
= \frac{P(z)}{G(z)}$ and
let $\tilde{r}(z,\bar{z}) = \sabs{G(z)}^2-\snorm{P(z)}^2$.
Assume that $\tilde{r}(0,0) = 1$.
By Lemma~\ref{lemma:targetautform}, we have that $\tilde{r}(z,\bar{z}) =
r(z,\bar{z})$.
Since $f$ of degree strictly greater than 1, Lemma \ref{lemma:Lnormal} shows $\Lambda_f$
has a unique critical point.
Since $\tilde{r} = r$, we find that $\Lambda_f= \Lambda_F$.
As $f$ is in normal form, the critical point of $\Lambda_f$ is at the origin.
Again from the previous lemma, $\Lambda_F = \Lambda_{f \circ \varphi^{-1}} = C
\Lambda_f \circ \varphi^{-1}$ for some constant $C$.  But this means that
$\varphi$ fixes this critical point, that is, $\varphi(0)=0$. An automorphism of the unit ball fixing the origin must be a unitary map (Corollary 1.6 of \cite{JPDrationalbook}), and hence $\varphi \in U(n)$.  Since $f$ is in normal form, fixes the
origin, and $\tau \circ f \circ \varphi^{-1} = f$, we find that $\tau$ also
fixes the origin and that $\tau \in U(N)$.

That $A_f$ is closed already follows from D'Angelo and Xiao \cites{DX1,DX2},
and it is also an immediate consequence of $A_f \leq U(n) \oplus U(N)$.
\end{proof}

In particular, the lemma above gives $\Gamma_f \leq U(n)$, but we can read
even more from the proof.
We thus have the following characterization of $\Gamma_f$, where no
rescaling is necessary.

\begin{lemma} \label{lemma:gammapreservesr}
Suppose $f = \frac{p}{g} \colon \bB_n \to \bB_N$ is a proper rational map
in normal form that is not linear.  Then
$U \in \Gamma_f \leq U(n)$ if and only if
\begin{equation}
\sabs{g(z)}^2-\snorm{p(z)}^2 \equiv
\sabs{g(Uz)}^2-\snorm{p(Uz)}^2 .
\end{equation}
In other words, $\Delta^{(*,*)}_f = \Gamma_f$.
Moreover, $\Gamma_f \leq \Delta^{(a,b)}_f$.
\end{lemma}

\begin{proof}
That $\Gamma_f \leq U(n)$ follows from the lemma above.
If $U \in \Gamma_f$, then there is a $V \in U(N)$ such that
$V \circ f \circ U = f$.  Plugging into the underlying form and clearing
denominators obtains
\begin{equation}
\sabs{g(U z)}^2-\snorm{p(U z)}^2 =
\sabs{g(U z)}^2-\snorm{V p(U z)}^2 =
\sabs{g(z)}^2-\snorm{p(z)}^2 .
\end{equation}
Conversely, suppose $\sabs{g(z)}^2-\snorm{p(z)}^2 \equiv
\sabs{g(Uz)}^2-\snorm{p(Uz)}^2$.  Then Lemma~\ref{lemma:targetautform}
says that $f = \psi \circ f \circ U$ for some automorphism $\psi$, which is
a unitary as $A_f \leq U(n) \oplus U(N)$.
\end{proof}

\begin{lemma} \label{lemma:DfSigmaf}
Suppose $f \colon \bB_n \to \bB_N$ is a proper rational map
of degree $d > 1$.
Then
$\Delta^{(2,0)}_f = \Sigma_f$ and
$\Delta^{(*,0)}_f = D_f$.
\end{lemma}

\begin{proof}
First, put $f = \frac{p}{g}$ into normal form where $g(0)=1$.
Complexify $\sabs{g(z)}^2-\snorm{p(z)}^2$, and note that if $U \in
\Delta^{(*,0)}_f$ then
\begin{equation}
g(Uz)\bar{g}(\overline{Uz})-p(Uz) \cdot \bar{p}(\overline{Uz})
=
g(z)\bar{g}(\bar{z})-p(z) \cdot \bar{p}(\bar{z}) .
\end{equation}
Now plug in $\bar{z}=0$ to find
\begin{equation}
g(Uz) = g(Uz)\bar{g}(0)-p(Uz) \cdot \bar{p}(0)
=
g(z)\bar{g}(0)-p(z) \cdot \bar{p}(0)
=
g(z) .
\end{equation}
Hence $U \in D_f$.  Next we note that
\begin{equation}
\bigl(\sabs{g(z)}^2-\snorm{p(z)}^2\bigr)_{(*,0)}
=
g(z)\bar{g}(0)-p(z) \cdot \bar{p}(0)
=
g(z) ,
\avoidbreak
\end{equation}
which implies that if $U \in \Delta^{(*,0)}_f$, then $U \in D_f$.

The statement for $\Sigma_f$ follows analogously by considering only
the quadratic terms.
\end{proof}

\begin{lemma}
Suppose $f \colon \bB_n \to \bB_N$ is a rational proper map in normal form
and $f$ is not linear.  Then
\begin{enumerate}[(i)]
\item
$G_f \leq \Gamma_f \leq D_f \leq \Sigma_f \leq U(n)$
and
$\Gamma_f \leq \Delta^{(a,b)}_f$ 
are all closed subgroups.
\item
$H_f \leq U(N)$ and $T_f \leq U(N)$ are closed subgroups.
\end{enumerate}
\end{lemma}

\begin{proof}
The containment in the unitary follows since $A_f \leq U(n) \oplus U(N)$.
That $A_f$, $\Gamma_f$, $G_f$, $T_f$, and $H_f$ are closed
was proved by D'Angelo and Xiao in their work, and it also follows rather
quickly once they are subgroups of the unitary group.
That $D_f$, $\Sigma_f$, $\Delta^{(a,b)}_f$ are closed follows as they are given
by an invariant polynomial.
The groups $\Gamma_f$, $G_f$, $T_f$ and $H_f$ are either equivalent to
subgroups of $A_f$
or the projections onto the first or the second factor.  In either case, it
follows that they are all subgroups of the correct unitary groups.

The inclusions $G_f \leq \Gamma_f$,  $\Gamma_f \leq
\Delta^{(a,b)}_f$, and
$D_f \leq \Sigma_f$ follow immediately.
That $\Gamma_f \leq D_f$ follows from Lemma~\ref{lemma:DfSigmaf}.
\end{proof}

Theorem~\ref{thm1} follows from the lemmas above.


\section{Characterization of \texorpdfstring{$\Gamma_f$}{Gammaf}} \label{section:charGammaf}

We prove Theorem~\ref{thm:char} in the following two lemmas.
First, we observe that Lemma \ref{lemma:gammapreservesr} quickly yields the first part of Theorem~\ref{thm:char}.

\begin{lemma}
Suppose that $f \colon \bB_n \to \bB_N$ is a rational proper map in normal
form and $f$ is not linear.  Then $\Gamma_f \leq U(n)$ is a subgroup that is
defined by a single invariant polynomial.
\end{lemma}

\begin{proof}
Let $f=\frac{p}{g}$ and
apply Lemma~\ref{lemma:gammapreservesr}.  In particular,
if $\rho(z,\bar{z}) = \sabs{g(z)}^2-\snorm{p(z)}^2$, then
the lemma implies that
\begin{equation}
\Gamma_f = \bigl\{ U \in U(n) : \rho(Uz,\overline{Uz}) = \rho(z,\bar{z})
\bigr\} .
\qedhere
\end{equation}
\end{proof}

We now prove the second part of Theorem~\ref{thm:char}. 

\begin{lemma}
Given any group $\Gamma \leq U(n)$ that is defined by finitely many
invariant polynomials, there exists a polynomial proper map
$f \colon \bB_n \to \bB_N$ such that $f(0)=0$ and $\Gamma_f = \Gamma$.
\end{lemma}

We note that the $N$ required depends on the inviariant polynomials given.

\begin{proof}
Suppose that $\Gamma \leq U(n)$ is the group given by the invariant polynomials
$\rho_1,\ldots,\rho_\ell$:
\begin{equation}
\Gamma = \bigl\{ U \in U(n) : \rho_j(Uz,\overline{Uz}) = \rho_j(z,\bar{z}),
j=1,\ldots,\ell
\bigr\} .
\end{equation}
We can ensure that $\rho_j(0)=0$ for each $j$.
We will construct a single polynomial that defines $\Gamma$.  Suppose that
$\rho_j$ is of bidegree $(d_j,d_j)$.  Write $k_j = d_1+d_2+\cdots+d_j$
and construct
\begin{equation}
\rho(z,\bar{z})
=
\snorm{z}^2
\sum_{j=1}^\ell
\snorm{z}^{2k_j}
\rho_j(z,\bar{z}) .
\end{equation}
Note that $\snorm{Uz}^{2} = \snorm{z}^2$ for any unitary $U$, and
that each polynomial
$\snorm{z}^{2k_j}
\rho_j(z,\bar{z})$ only has monomials of degree $2k_{j-1}+1$ through
$2_{k_j}$.  In particular, no two
$\snorm{z}^{2k_j} \rho_j(z,\bar{z})$ have monomials of the same degree.
Therefore, $\rho(Uz,\overline{Uz}) = \rho(z,\bar{z})$ if and only if
$\rho_j(Uz,\overline{Uz}) = \rho_j(z,\bar{z})$ for each $j$.
The first factor $\snorm{z}^2$ ensures that $\rho$ has no holomorphic or
antiholomorphic terms.
Suppose that $\rho$ is of bidegree $(d,d)$.  Consider
\begin{equation}
R(z,\bar{z})
=
\sum_{j=1}^d
\frac{1}{d} \snorm{z}^{2j} .
\end{equation}
We have that $R(z,\bar{z})=1$ when $\snorm{z}^2=1$ and moreover its matrix
of coefficients is positive definite as it is a sum of squares of
holomorphic polynomials.  These polynomials thus give a polynomial proper map
of balls.  For a small $\epsilon > 0$ write
\begin{equation}
R_\epsilon(z,\bar{z})
= R(z,\bar{z}) + \epsilon \rho(z,\bar{z}) (1-\snorm{z}^2) .
\end{equation}
Note that the matrix of coefficients of $R_\epsilon$ has zeros at all
the entries for holomorphic or antiholomorphic terms.  Moreover, except for
the constant, the diagonal terms of the matrix for $R$ are positive and at least
$\frac{1}{d}$.  Thus, for a small enough $\epsilon$, the matrix
for $R_\epsilon$ is still positive definite and hence a sum of squares of
holomorphic polynomials
\begin{equation}
R_\epsilon(z,\bar{z}) =
\sabs{f_1(z)}^2+\cdots+\sabs{f_N(z)}^2 .
\end{equation}
Since the matrix of coefficients is zero at the entries for holomorphic
and antiholomorphic terms, the polynomials $f_j$ can be picked so that $f_j(0)=0$ for
all $j$.
As $R_\epsilon = 1$ when $\snorm{z}=1$, we have a proper map of balls
$f=(f_1,\ldots,f_N)$.
We have that $U \in \Gamma_f$ if and only if
\begin{multline}
R_\epsilon(Uz,\overline{Uz})
= R(Uz,\overline{Uz}) + \epsilon \rho(Uz,\overline{Uz}) (1-\snorm{Uz}^2) 
\\
= R(z,\overline{z}) + \epsilon \rho(Uz,\overline{Uz}) (1-\snorm{z}^2) 
= R_\epsilon(z,\bar{z}) .
\end{multline}
That is, $U \in \Gamma_f$ if and only if $\rho(Uz,\overline{Uz})=\rho(z,\bar{z})$
and that defines $\Gamma$.  The map $f$ is the desired map.
\end{proof}


\section{Constructions} \label{section:constructions}

The following basic result is useful for constructing maps with
a given denominator.  The result we prove below is a straightforward
construction that gives a numerator of a specific degree, for denominators
close to 1.
There is a far deeper result, see 
D'Angelo~\cite{JPDrationalbook} or
Catlin--D'Angelo~\cite{CatlinDAngelo},
that any polynomial that does not vanish on the closed ball
is a denominator of a proper map of balls.  In that case however,
the degree cannot be bounded.

\begin{prop} \label{prop:constructwithdenom}
Given any polynomial $G \colon \C^n \to \C$ of degree $d-1$ such that
$G(0)=0$, there exists an $\epsilon > 0$, $N \in \N$,
and a polynomial $P \colon \C^n \to \C^N$ of degree $d$
such that $P(0)=0$ and $\frac{P}{1+\epsilon G}$ is a proper map
of $\bB_n$ to $\bB_N$.  The $N$ can be taken to be one less than
the number of different monomials of degree $d$ or less in $n$ variables.
\end{prop}

A proof for $d=3$ was given in \cite{Lnormal} and must be somewhat modified
for higher degree.

\begin{remark}
It may be convenient in some constructions to consider $1+G(\epsilon z)$ instead of
$1+\epsilon G(z)$, and the proof follows in exactly the same way.  That is, given
an affine variety, there is a proper map with this variety as the pole set provided we can
we can dilate it sufficiently far away from the origin.
\end{remark}

\begin{proof}
One starts with the polynomial
\begin{equation} 
R(z,\bar{z}) = \sum_{j=1}^d \frac{1}{d} \snorm{z}^{2j} .
\end{equation}
If we ignore the row and column corresponding to purely holomorphic and
purely antiholomorphic terms (which is a single row and column), its
matrix of coefficients is positive definite.  
Set $N$ to be the rank of this matrix, which is one less than 
the number of all holomorphic monomials in $n$ variables of degree $d$ or less.
Moreover $R=1$ when $\snorm{z}=1$.
Now consider
\begin{equation}
r(z,\bar{z}) = \epsilon \bigl(G(z)+\overline{G(z)}\bigr)\bigl(1-\snorm{z}^2\bigr)
-R(z,\bar{z}) + 1.
\end{equation}
The rank of this matrix is $N+1$, as we will have a $1$ on the coefficient
of the matrix for the constant term, the off diagonal elements are all
small if $\epsilon > 0$ is small, and all other diagonal elements are negative
and of size roughly $\frac{1}{d}$ or larger (assuming $\epsilon$ is small).
Thus, for a small $\epsilon$, the matrix of coefficient will have 1 positive
and $N$ negative eigenvalues.
The matrix of coefficients for
\begin{equation}
r(z,\bar{z})-\sabs{1+\epsilon G(z)}^2
\end{equation}
has zeros at all the terms corresponding to the pure holomorphic and pure
antiholomorphic terms, and hence is of rank $N$ and therefore negative
semidefinite.  In particular, there exists a polynomial map $P \colon \C^n
\to \C^N$ such that
\begin{equation}
r(z,\bar{z}) = \sabs{1+\epsilon G(z)}^2-\snorm{P(z)}^2 .
\end{equation}
The degree of $P$ is at most $d$ as that is the maximal degree of all
monomials, and as there were no holomorphic or antiholomorphic terms in the
matrix from which we constructed $P$, we can chose $P$ to not include the
constant monomial, and hence $P(0)=0$.  In fact, since 
$\snorm{1+\epsilon G(z)}^2$ has no terms of bidegree $(d,d)$ and $r$ does,
we must conclude that $P$ is in fact of degree $d$ and not less.
Since we also get that $r(z,\bar{z})=0$ when $\snorm{z}=1$, we find that
$\frac{P}{1+\epsilon G}$ is a proper map.

It is left to show that the map is in lowest terms.  If not, that is,
if there was a common multiple $h$ of the components of $P$ and $1+\epsilon G$,
then we would have $r(z,\bar{z}) = \sabs{h(z)}^2 A(z,\bar{z})$ for some
real polynomial $A$.
Consider the top degree part of $r$,
that is, we look at the bidegree $(d,d)$ part of $r$ which is simply
\begin{equation}
r(z,\bar{z})_{(d,d)} = - \frac{1}{d} \snorm{z}^{2d} .
\end{equation}
The function $h$, were it nonconstant, would be zero at some points arbitrarily far
away from the origin, while $r_{(d,d)}$ and hence $r$ must become
arbitrarily negative as $\snorm{z}$ becomes large.  So $h$ is a constant
and $\frac{P}{1+\epsilon G}$ is in lowest terms.
\end{proof}

Using the previous proposition, we can
construct a map of degree 4 with denominator
\begin{equation} \label{eq:trivdenom}
1+\sigma_1 z_1^2 + \cdots + \sigma_n z_n^2 + \epsilon_1
z_1^3+\cdots+\epsilon_n z_n^3 ,
\end{equation}
which then necessarily has the trivial group $D_f = \{ I \}$, and hence
$\Gamma_f = \{ I \}$.

If the denominator is of the form
$1+\sigma_1 z_1^2 + \cdots + \sigma_n z_n^2$,
then $D_f = \{ I , -I \}$.  If we
want $G_f$ to also be $\{ I, -I \}$, then
we would need the numerator to be invariant, but that requires degree 4, see Lemma~\ref{lemma:Gfgenericsigma}.
For
small $\sigma$ we can always construct such a
degree 4 map.

\begin{prop} \label{prop:constructZ2}
\pagebreak[2]
Given $n \in \N$, there exists an $N$ and an $\epsilon > 0$
such that whenever $0 < \sigma_1 < \cdots < \sigma_n < \epsilon$
then there exists a rational proper map of balls of degree 4 in normal form
$f=\frac{p}{g} \colon \bB_n \to \bB_N$ (in lowest terms) such that
\begin{equation}
g(z) = 1 + \sum_{j=1}^n \sigma_j z_j^2
\avoidbreak
\end{equation}
and such that $p(z)=p(-z)$.  In other words, $G_f = \{ I, -I \}$.
\end{prop}

We note that the construction in Example~\ref{example:explicit}
can be adapted to arbitrary $n$ and ensure the existence of such an
example whenever $\sigma_1^2+\cdots+\sigma_n^2 < 1$, so
$\epsilon = \frac{1}{\sqrt{n}}$ would suffice.
The advantage of the approach given here is that it easily generalizes to more
complicated denominators.

\begin{proof}
We adapt the proof of Proposition~\ref{prop:constructwithdenom} from above, however we need to work
with forms where only holomorphic and antiholomorphic monomials of even
degree arise.
Let $N$ be the number of monomials in $n$ variables of degree $2$ and $4$.
Thus start with
\begin{equation}
R(z,\bar{z}) = \frac{1}{2}\snorm{z}^4+\frac{1}{2}\snorm{z}^8 .
\end{equation}
Again $R=1$ if $\snorm{z}=1$.
Write $G(z) = \sum_{j=1}^n \sigma_j z_j^2$ and construct
\begin{equation}
r(z,\bar{z}) = \bigl( G(z)+\overline{G(z)}\bigr)\bigl(1-\snorm{z}^4\bigr)
-R(z,\bar{z}) + 1.
\end{equation}
Note the $\snorm{z}^4$ in the formula.  The matrix of coefficients has
nonzero rows and columns only for monomials of even degree, and taking this
submatrix we find a full rank matrix of rank $N+1$.
As the on diagonal elements are roughly of size $\frac{1}{2}$ or larger
as long as $\sigma$s are small, and all off diagonal elements are of size
proportional to the $\sigma$s, we find that the number of negative
eigenvalues is $N$ and there is one positive eigenvalue corresponding
the constant $1$.
The matrix of
$\sabs{1+G(z)}^2$ also only has nonzero rows and columns for the monomials
of even degree, and hence so does
\begin{equation}
r(z,\bar{z})-\sabs{1+G(z)}^2 ,
\end{equation}
whose matrix again has no elements for the row and column corresponding
to the constant.  Hence its rank is $N$ and it must be a negative
semidefinite matrix so there is a polynomial $p(z)$ only using
degree $2$ and $4$ monomials such that
\begin{equation}
r(z,\bar{z})=\sabs{1+G(z)}^2 -\snorm{p(z)}^2 .
\end{equation}
Again $r=0$ on $\snorm{z}=1$ and we are finished,
$\frac{p}{1+G}$ is the desired map.  It is in lowest terms by the same
argument as in
Proposition~\ref{prop:constructwithdenom}.
\end{proof}


\section{Group invariance of maps with generic \texorpdfstring{$\sigma$}{sigma}} \label{section:gensigmas}

We prove the first part of Corollary~\ref{cor1} in the next lemma.
The construction part of the corollary we have
already done in Proposition~\ref{prop:constructZ2}.

\begin{lemma}\label{lemma:Gfgenericsigma}
Suppose $f = \frac{p}{g} \colon \bB_n \to \bB_N$
is a rational proper map in normal form
such that $0 < \sigma_1 < \ldots < \sigma_n$.
Then $G_f$ is either trivial, or $G_f = \{ I, -I \}$, in which
case $p(z) = p(-z)$ and $g(z)=g(-z)$.  In particular,
if $G_f = \{ I, -I \}$ then the degree of $f$ is at least 4.
\end{lemma}

\begin{proof}
First, $G_f \leq \Sigma_f$, and as the $\sigma$s are nonzero and distinct,
we have that $\Sigma_f$ is composed of diagonal matrices with $\pm 1$ on the
diagonal.
Via the theorem of Lichtblau~\cite{Li}, $G_f$ must be fixed-point-free and
cyclic, and the only such subgroups are $\{ I \}$ or $\{I, -I \}$.
Suppose that $G_f = \{ I, -I \}$.
As $G_f \leq D_f$, we have that $\{I,-I\} \leq D_f$ and so
$g(z)=g(-z)$.
We have
\begin{equation}
\frac{p(z)}{g(z)} = \frac{p(-z)}{g(-z)}
\quad\text{or}\quad
p(z)g(-z) = p(-z)g(z) .
\end{equation}
As $g(z)=g(-z)$ we find that $p(z)=p(-z)$.
Thus all the monomials that appear in $g$ and $p$ are of even degree.
The degree of $p$ cannot be $2$ as normal form implies that $\deg(g) <
\deg(p)$, hence $\deg(p) \geq 4$.
\end{proof}

It is worth remarking that the condition $G_f=\{I, -I\}$ immediately gives that $f$ is even, but in the Corollary, we further prove that the numerator and denominator must also be even for ball maps.

The second part of the corollary now follows
from Proposition~\ref{prop:constructZ2}.  The proposition says that the
given map satisfies $G_f \geq \{I , -I \}$, and the lemma says that this
must be an equality.

Let us give an explicit example in a slightly different way for $\bB_2$.
The example is a modification of an example given by
Al Helal~\cite{helal:prelim}.  An advantage of this construction
is that it allows explicit bounds on the $\sigma$'s.  On the other hand, it
does not generalize easily to more complicated denominators.

\begin{example} \label{example:explicit}
We will give this example in the $n=2$ case, but it is easy to
generalize to higher $n$.
We start with an automorphism $\varphi$ of $\bB_3$ as given by \eqref{e:auto}.
We will pick $\alpha = \langle - \sigma_1 , 0 , - \sigma_2 \rangle$,
and as $\alpha \in \bB_3$ we require that $\sigma_1^2+\sigma_2^2 < 1$.
Consider the homogeneous proper ball map $H \colon \bB_2 \to \bB_3$
given by
\begin{equation}
H(z_1,z_2) = \bigl(z_1^2,\sqrt{2}\,z_1z_2,z_2^2\bigr) .
\end{equation}
Write $\varphi \circ H$ as
\begin{equation}
\varphi \circ H (z) = \psi(z) = \bigl( \psi_1(z), \psi_2(z), \psi_3(z) \bigr) .
\end{equation}
Note that the denominator of this map is precisely
\begin{equation}
1 + \sigma_1 z_1^2 + \sigma_2 z_2^2 .
\end{equation}
The map is not in normal form; while
$\psi_2(0)=0$, we have $\psi_1(0) \not= 0$ and $\psi_3(0)\not= 0$.
By tensoring $\psi_1$ and $\psi_2$ by $H$, we get a new map that is in normal form; namely, we
consider the map
\begin{equation}
f = \bigl( ( \psi_1 \oplus \psi_3 ) \otimes H \bigr) \oplus \psi_2 .
\end{equation}
Then $f \colon \bB_2 \to \bB_7$ is a degree 4 map in normal form
with the desired denominator.
The map is invariant under $-I$ as all the monomials that appear are
quadratic and hence $G_f = \{ I, -I \}$.
\end{example}

\section{Acknowledgements}

The authors would like to thank the anonymous referee for their careful reading and helpful comments, which significantly improved the quality of this paper. The second author was in part supported by Simons Foundation collaboration grant 710294.




\def\MR#1{\relax\ifhmode\unskip\spacefactor3000 \space\fi%
  \href{http://mathscinet.ams.org/mathscinet-getitem?mr=#1}{MR#1}}


\end{document}